\documentclass[11pt]{amsart}

\usepackage{amsmath, amsthm, amssymb, amscd, amsfonts, enumerate}
\usepackage{fontsmpl}
\usepackage{fontsmpl,layout}
\usepackage[all]{xy}
\usepackage[dvips, dvipsnames, usenames]{color}
\usepackage{dsfont}
\usepackage[latin1]{inputenc}
\usepackage{hyperref}
\usepackage{bbm}

\usepackage[active]{srcltx}

\numberwithin{equation}{section}

\newtheorem{lema}{Lemma}[section]
\newtheorem{theorem}[lema]{Theorem}

\newtheorem{prop}[lema]{Proposition}

\newtheorem*{maintheorem}{Main Theorem}

\theoremstyle{definition}
\newtheorem{definition}[lema]{Definition}

\theoremstyle{remark}
\newtheorem{obs}[lema]{Remark}

\theoremstyle{plain}

\newcommand{\nc}{\newcommand}

\renewcommand{\_}[1]{\mbox{$_{\left( #1 \right)}$}}

\newcommand{\supp}{\operatorname{supp}}

\nc{\im}{\mathtt{i}}

\newcommand{\hmh}{{}_H\hspace{-1pt}{\mathcal M}_H}
\newcommand{\hm}{{}_H\hspace{-1pt}{\mathcal M}}

\newcommand{\kgmkg}{{}_{\ku^G}\hspace{-1pt}{\mathcal M}_{\ku^G}}

\nc{\RR}{{\mathbb R}} \nc{\CC}{{\mathbb C}} \nc{\ZZ}{{\mathbb Z}}
\nc{\FF}{{\mathbb F}} \nc{\NN}{{\mathbb N}} \nc{\QQ}{{\mathbb Q}}
\nc{\PP}{{\mathbb P}} \nc{\DD}{{\mathbb D}} \nc{\Sn}{{\mathbb S}}
\nc{\uno}{\mathbb{1}}

\nc{\vs}{\vspace{.5cm}}

\nc{\no}{\smallbreak\noindent}

\newcommand\id{\operatorname{id}}
\newcommand\ad{\operatorname{ad}}

\newcommand\gr{\operatorname{gr}}

\newcommand\co{\operatorname{co}}

\newcommand\End{\operatorname{End}}

\newcommand\Ind{\operatorname{Ind}}

\newcommand{\eps}{\varepsilon}
\newcommand{\ot}{\otimes}

\newcommand\sgn{\operatorname{sgn}}

 \nc{\D}{\Delta} \nc{\e}{\varepsilon}
\nc{\GL}{\operatorname{GL}} \nc{\wact}{\rightharpoonup}
\nc{\Tr}{\mathrm{Tr}} \nc{\cark}{char\,k} \nc{\adl}{\ad_\ell}
\nc{\ku}{\Bbbk}

\newcommand{\ydg}{{}_{\ku G}^{\ku G}\mathcal{YD}}
\newcommand{\ydgdual}{{}_{\ku^G}^{\ku^G}\mathcal{YD}}

\newcommand{\ydh}{{}^H_H\mathcal{YD}}
\newcommand{\ydhdual}{{}^{H^*}_{H^*}\mathcal{YD}}

\newcommand{\ydsn}{{}^{\ku{\Sn_n}}_{\ku{\Sn_n}}\mathcal{YD}}
\newcommand{\ydsnd}{{}^{\ku^{\Sn_n}}_{\ku^{\Sn_n}}\mathcal{YD}}

\newcommand{\ydst}{{}^{\ku^{\Sn_3}}_{\ku^{\Sn_3}}\mathcal{YD}}

\nc{\yd}{\mathcal{YD}}

\nc{\ydSndual}[1]{{}^{\ku^{\Sn_{#1}}}_{\ku^{\Sn_{#1}}}\yd}

\nc{\cP}{\mathcal{P}} \nc{\cU}{\mathcal{U}}
\nc{\cE}{\mathcal{E}} \nc{\cS}{\mathcal{S}}
\nc{\cC}{\mathcal{C}} \nc{\cO}{\mathcal{O}} \nc{\cQ}{\mathcal{Q}}
\nc{\cB}{\mathcal{B}} \nc{\cJ}{\mathcal{J}} \nc{\cI}{\mathcal{I}}

\nc{\fD}{\mathfrak{D}} \nc{\fI}{\mathfrak{I}} \nc{\fJ}{\mathfrak{J}}
\nc{\fS}{\mathfrak{S}}

\nc{\Ho}{H_0} \nc{\GH}{G(H)} \nc{\mas}{\oplus}
\nc{\coM}{\mathcal{M}^\ast(2,k)} \nc{\cA}{\mathcal{A}}
\nc{\PH}{\cP(H)}
\nc{\Ftwist}{\overset{\curvearrowright}F}

\nc{\rep}{{\mathcal Rep}(H)}

\newcommand{\BV}{\mathfrak{B}}

\newcommand{\xij}[1]{x_{(#1)}}

\newcommand{\dij}[1]{\delta_{#1}}

\newcommand{\Cg}[1]{C_{G}(#1)}

\newlength{\ancho}
\newlength{\alto}

\newcommand{\ao}[1]{
\settowidth{\ancho}{$#1$}
\settoheight{\alto}{$#1$}
\hbox{\vbox{  \hbox{{${}_a$}}\kern \alto}\vbox{\hrule height 0.5pt width \ancho \kern 1mm \hbox{$#1$}}}
}

\newcommand{\bo}[1]{
\settowidth{\ancho}{$#1$}
\settoheight{\alto}{$#1$}
\hbox{\vbox{ \hbox{${}_b$}\kern \alto   }\vbox{\hrule height 0.5pt width \ancho \kern 1mm \hbox{$#1$}}}
}

\begin{document}

\title[Hopf algebras with coradical $\ku^{\Sn_3}$]
{Finite dimensional Hopf algebras over the dual group algebra of the symmetric group in three letters}

\author[andruskiewitsch and vay]
{Nicol\'as Andruskiewitsch and Cristian Vay}

\address{FaMAF-CIEM (CONICET), Universidad Nacional de C\'ordoba,
Medina A\-llen\-de s/n, Ciudad Universitaria, 5000 C\' ordoba, Rep\'
ublica Argentina.} \email{andrus@famaf.unc.edu.ar,
vay@famaf.unc.edu.ar}

\thanks{This work was partially supported by ANPCyT-Foncyt, CONICET, Ministerio de Ciencia y
Tecnolog\'{\i}a (C\'ordoba)  and Secyt (UNC). Part of the work of C. V. was done as a fellow of the Erasmus Mundus programme of the EU in the University of Antwerp. He thanks to Prof. Fred Van Oystaeyen for his warm hospitality and help.}

\date{\today}
\dedicatory{Dedicated to Mia Cohen with affection}
\begin{abstract}
We classify finite-dimensional Hopf algebras whose coradical is isomorphic to the algebra of functions on $\Sn_3$.
We describe a new infinite family of Hopf algebras of dimension 72.
\end{abstract}

\maketitle



\section{Introduction}

This work is a contribution to the  classification of finite-dimensional Hopf algebras over an algebraically closed field $\ku$ of characteristic $0$, problem posed by I. Kaplansky in 1975. We are specifically interested in Hopf algebras whose coradical is the dual of a non-abelian group algebra.
We note that there are very few classification results on finite-dimensional Hopf algebras whose coradical is a Hopf subalgebra but not a group algebra, a remarkable exception being  \cite{de1tipo6chevalley}.

Let $A$ be a Hopf algebra whose coradical $H$ is a Hopf subalgebra. It is well-known that the associated
graded Hopf algebra of $A$ is isomorphic to $R\#H$ where
$R = \oplus_{n\in \mathbb N_0}R^n$ is a braided Hopf algebra in the category $\ydh$ of Yetter-Drinfield modules over $H$. As explained in \cite{andrussch3}, to classify finite-dimensional Hopf algebras $A$ whose coradical is isomorphic to $H$ we have to deal with the following questions:

\begin{enumerate}\renewcommand{\theenumi}{\alph{enumi}}\renewcommand{\labelenumi}{(\theenumi)}
  \item\label{que:nichols-fd} Determine all Yetter-Drinfield modules $V$ over $H$ such that the Nichols algebra $\BV(V)$ has finite dimension; find an efficient set of relations for $\BV(V)$.

  \item\label{que:nichols-R} If $R = \oplus_{n\in \mathbb N_0}R^n$ is a finite-dimensional Hopf algebra in $\ydh$ with $V = R^1$, decide if $R \simeq \BV(V)$ (generation in degree one).

  \item\label{que:lifting} Given $V$ as in \eqref{que:nichols-fd}, classify all $A$ such that $\gr A \simeq \BV(V)\#H$ (lifting).
\end{enumerate}

Now the category $\ydhdual$ is braided equivalent to $\ydh$ see e.g. \cite[2.2.1]{andrusgrania}.
Therefore, the answers to the questions \eqref{que:nichols-fd} and \eqref{que:nichols-R} in $\ydhdual$ give the analogous answers in $\ydh$.
Thus, if $H = \ku^G$ is the algebra of functions on a finite group $G$, then we are reduced to know the solutions to questions \eqref{que:nichols-fd} and \eqref{que:nichols-R} for the group algebra $\ku G$. In the case $G = \Sn_3$, we know that
\begin{itemize}
  \item there is a unique $V\in \ydst$ such that $\dim\BV(V)<\infty$, namely $V = M((12),\sgn)$ \cite[Thm. 4.5]{AHS}.
  \item If $R = \oplus_{n\in \mathbb N_0}R^n$ is a finite-dimensional Hopf algebra in $\ydst$, then $R \simeq \BV(V)$ \cite[Thm. 2.1]{andrusgrania3}.
\end{itemize}

In this paper, we solve question \eqref{que:lifting} for $\ku^{\Sn_3}$. We introduce a family of Hopf algebras $\cA_{[a_1,a_2]}$, $(a_1,a_2)\in \ku^2$;
these are new as far as we know, except for $(a_1,a_2) = (0,0)$.
Let $\Gamma = \ku^{\times}\times\Sn_3$, where $\ku^{\times}$ is the group of the invertible elements of $\ku$.
We consider the right action $\triangleleft$ of the group  on $\ku^2$ defined by
\begin{equation}\label{equ:action}
(a_1,a_2)\triangleleft (\mu, (12)) =\mu(a_2,a_1),\quad (a_1,a_2)\triangleleft( \mu, (123))=-\mu(a_2,a_2-a_1).
\end{equation}
We denote by $[a_1,a_2]\in \Gamma\backslash\ku^2$ the equivalence class of $(a_1,a_2)$ under this action; notice that $\Gamma\backslash\ku^2$
is infinite.
We prove in Theorem \ref{main result}:

\begin{maintheorem}
The set of isomorphism classes of finite-dimensional non-semisimple Hopf algebras with coradical $\ku^{\Sn_3}$ is
in bijective correspondence with $\Gamma\backslash\ku^2$ via the assignment $[a_1,a_2]\leftrightsquigarrow\cA_{[a_1,a_2]}$.
\end{maintheorem}

To show that the Hopf algebras $\cA_{[a_1,a_2]}$ have the right dimension we use the Diamond Lemma. To prove that any
finite-dimensional non-semisimple Hopf algebra $A$ with coradical $\ku^{\Sn_3}$ belongs to this family,
we describe the first term $A_1$ of the coradical filtration using \cite[Thm. 5.9.c)]{ardizzonimstefan}.

\subsection*{Acknowledgment} N. A. thanks Margaret Beattie for conversations on this problem in St. John's, May 2008.

\section{Preliminaries}\label{Sec: preliminaries}
\subsection{Notation} We fix an algebraically closed field $\ku$ of characteristic zero. If $V$ is a vector space,
$T(V)=\oplus_{n\geq0}V^{\ot n}$ is the tensor algebra of $V$. If there is no place for confusion,
we write $a_1\cdots a_n$ instead of $a_1\ot\cdots\ot a_n$ for $a_1,...,a_n\in V$.
If $G$ is a group, we denote by $e$ its identity element, by $\ku G$ the group algebra of $G$ and by $\ku^G$ the
dual group algebra, that is the algebra of functions on $G$.  We denote by $\Sn_n$ the symmetric group on $n$ letters.
We use Sweedler notation but dropping the summation symbol. If $H$ is an algebra, then we denote by $\hm$, resp. $\hmh$,
the category of left $H$-modules, resp. $H$-bimodules.

\subsection{The coradical filtration}
Let $C$ be a coalgebra with comultiplication $\D$. The coradical $C_0$ of $C$ is the sum of all simple subcoalgebras of $C$.
The coradical filtration of $C$ is the family of subspaces defined inductively by $C_n:=C_0\wedge C_{n-1}=\D^{-1}(C\ot C_0+C_{n-1}\ot C)$
for each $n\geq1$. Then
\begin{equation}\label{eq: properties of coradical filtration}
C_n\subseteq C_{n+1},\quad C=\bigcup_{n\geq0}C_n\quad\mbox{and}\quad\D(C_n)\subseteq\sum_{i=0}^nC_i\ot C_{n-i},
\end{equation}
for all $n\geq0$ \cite[Thm. 5.2.2]{mongomeri}. We denote by $\gr C=\oplus_{n\geq0}\gr^n C$ the associated graded coalgebra of $C$;
as vector space $\gr^n C:=C_n/C_{n-1}$ for all $n\geq0$ where $C_{-1}=0$. We denote by $G(C)$ the group of group-like elements of $C$.
As usual, for $g,h\in G(C)$, $\cP_{g,h}(C)=\{x\in C_1|\D(x)=g\ot x+x\ot h\}$ is the space of $(g,h)$-skew primitive elements of $C$.
If $C$ is a bialgebra and $g=h=1$,
we write simply $\cP(C)$ instead of $\cP_{1,1}(C)$ and $x\in\cP(C)$ is called a primitive element.
A coalgebra $E$ is called a \emph{matrix coalgebra} of rank $n$ if it has a basis $(e_{ij})_{1\le i,j \le n}$ such that the
comultiplication and counit are defined by $\D(e_{ij})=\sum_{k=1}^{n} e_{ik}\ot e_{kj}$ and $\e(e_{ij})=\delta_{ij}$ for all $1\le i,j \le n$.

\begin{lema}\label{le:similar to trivial casiprimitive}
Let  $D=\ku g$ and $E=\langle e_{ij}|1\leq i,j\leq n\rangle$ be matrix coalgebras of rank 1 and $n$ respectively.
If $(x_i)_{i=1}^n\subset D\oplus E$  (direct sum of coalgebras) are such that
\begin{align*}
\D(x_i)=x_i\ot g+\sum_{j=1}^ne_{ij}\ot x_j,
\end{align*}
then there exist $a_1, ..., a_n\in\ku$ such that $x_i=a_ig-\sum_{j=1}^na_{j}e_{ij}$, $1\le i \le n$.
\end{lema}

\begin{proof}
Write $x_i=a_ig+\sum_{s, t=1}^{n}\alpha_{st}^ie_{st}$ for all $1 \le i\le n$, where $a_i, \alpha_{st}^i \in\ku$, $1 \le i, s, t\le n$. Now we calculate $\D(x_i)$ in two ways:

\begin{align*}\D(x_i) &=\D\biggl(a_ig+\sum_{s, t=1}^{n}\alpha_{st}^ie_{st}\biggr)=a_ig\ot g+\sum_{s, t, l=1}^n\alpha_{st}^ie_{sl}\ot e_{lt} \qquad \text{ and }
\\ \D(x_i) &=\biggl(a_ig+\sum_{s, t=1}^{n}\alpha_{st}^ie_{st}\biggr)\ot g+\sum_{j=1}^ne_{ij}\ot\biggl(a_jg+\sum_{s, t=1}^{n}\alpha_{st}^je_{st}\biggr)\\
&=a_ig\ot g + \sum_{s, t=1}^{n}\alpha_{st}^ie_{st} \ot g +\sum_{j=1}^na_je_{ij}\ot g +\sum_{s, t, j=1}^{n}\alpha_{st}^je_{ij}\ot e_{st}\\
&=a_ig\ot g + \sum_{\substack{s, t =1 \\ s\neq i}}^{n}\alpha_{st}^ie_{st} \ot g + \sum_{t=1}^n(a_t+\alpha_{it}^i)e_{it}\ot g
 + \sum_{s, t, j=1}^{n}\alpha_{st}^je_{ij}\ot e_{st}
\end{align*}
Then the second and third terms are zero. Thus $\alpha_{st}^i=0$ and $\alpha_{it}^i=-a_t$, for $1 \le i, s, t\le n$, $s\neq i$. Hence $x_i=a_ig+\sum_{s,t=1}^n\alpha_{st}^ie_{st}=a_ig-\sum_{t=1}^na_{t}e_{it}.$
\end{proof}

\subsection{Yetter-Drinfeld modules over the dual of a group algebra}\label{sec: yetter-drinfield modules}

Let $H$ be a finite-dimensional Hopf algebra and $\ydh$ be the category of left Yetter-Drinfeld modules over $H$,
see e.g. \cite{mongomeri}.
The category $\ydhdual$ is braided equivalent to $\ydh$ by the following recipe.
Let $(h_i)$ and $(f_i)$ be dual basis of $H$ and $H^*$. Let $V\in\ydh$, $v\in V$ and $f\in H^*$, then $V$ turns into a Yetter-Drinfeld module over $H^{*}$ by
\begin{equation}\label{eq:dual-yd}
f\cdot v =\langle \cS(f),v\_{-1}\rangle v\_{0}, \quad \lambda(v) =\sum_i \cS^{-1}(f_i)\ot h_i\cdot v, \quad f\in H^*, v\in V.
\end{equation}
This gives a braided equivalence between $\ydh$ and $\ydhdual$ \cite[2.2.1]{andrusgrania}.

\bigbreak
Let $G$ be a finite group.
If $g\in G$, then we denote by $\cO_g$ the conjugacy class of $g$ and by $\Cg{g}$ the centralizer of $g$.
\begin{definition}\label{eq:definition of simple in YD group}
Fix $g\in G$ and $(\rho,V)$ an irreducible representation of $\Cg{g}$. Then
\begin{align*}
 M(g,\rho) & :=\Ind_{\Cg{g}}^{G} V=\ku G\ot_{{\Cg{g}}} V=\cO_g\ot V
\end{align*}
is an object of $\ydg$ in the following way. Let $(w_{i})_{1\le i \le r}$ be a basis of $V$ and $(h_{j})_{1\le j \le s}$ be a set of representatives of $G/\Cg{g}$, so that $(h_j\ot w_i)_{ji}$ is a basis of $M(g,\rho)$.
 Let $t_j = h_j g h_j^{-1}$, $1\le j \le s$. The action and coaction on $M(g,\rho)$ are given by
\begin{align*}
&h\cdot (h_j\ot w_i)=h_{k}\ot\rho(\tilde g)(w_i) &\mbox{ and }&& \delta(h_j\ot w_i)=t_j\ot  (h_j\ot w_i),
\end{align*}
$1\le i \le r$, $1\le j \le s$ and $h\in G$, where $hh_j=h_{k}\tilde g$ for unique $k$, $\tilde g\in\Cg{g}$.
\end{definition}
Let $\cQ$ be a set of representatives of conjugacy class of $G$.
It is well-known that $M(g,\rho)$ is simple and that any simple object of $\ydg$ is isomorphic
to $M(g,\rho)$ for unique $g\in\cQ$ and a unique isomorphism class $[(\rho,V)]$ of irreducible representations of $\Cg{g}$.
By the braided equivalence described above, see \eqref{eq:dual-yd}, this gives a parametrization of the irreducible objects in $\ydgdual$. Explicitly, $M(g,\rho)\in\ydgdual$ with
\begin{align}\label{eq: yetter-drinfield over de dual of G}
\delta_h\cdot(h_j\ot w_l)=\delta_{h,t_j^{-1}} h_j\ot w_l, \,
\lambda(h_j\ot w_l)=\sum_{h\in G}\delta_{h^{-1}}\ot h\cdot(h_j\ot w_l).
\end{align}

\subsection{Nichols algebras}
Our reference for Nichols algebras is \cite[Subsection 2.1]{andrussch3}. Let $V\in \ydh$. The Nichols algebra $\BV(V) = \oplus_{n\in \mathbb N_0}\BV^n(V)$ of $V$ is the braided graded Hopf algebra in $\ydh$ which satisfies:
\begin{itemize}\label{properties of BV}
\item $\BV^0(V)=\ku1$ and $\cP(\BV(V))=\BV^1(V)\simeq V$ in $\ydh$.
\item $\BV(V)$ is generated as an algebra by $\BV^1(V)$.
\end{itemize}

The tensor algebra $T(V)$ is a braided graded Hopf algebra in $\ydh$ such that $\D(v)=v\ot 1+1\ot v$ and $\e(v)=0$ for all $v\in V$.
Then the Nichols algebra can be described as $\BV(V)=T(V)/\cJ$,  where
$\cJ$ is the largest Hopf ideal of $T(V)$ generated by homogeneous elements of degree $\geq 2$.

\begin{lema}\label{le:component of degree 2 of ideal which defines the Nichols}
If $x\in\cJ^2=\cJ\cap V^{\ot2}$, then $x\in\cP(T(V))$.
\end{lema}

\begin{proof}
We have $\D_{T(V)}(x)=x\ot1+1\ot x+y$ with $y\in V\ot V\cap (\cJ\ot T(V)+T(V)\ot \cJ)$. Since $\cJ\subset\oplus_{n\geq2}V^{\ot n}$, $y=0$.
\end{proof}

\subsection{Hopf algebras whose coradical is a Hopf subalgebra}\label{subsection:coradical=Hopf}
Let $H$ be a semisimple and cosemisimple Hopf algebra (here char $\ku =0$ is not needed).
Let $A$ be a Hopf algebra whose coradical $A_0$ is a Hopf subalgebra isomorphic to $H$.
By \cite[5.2.8]{mongomeri}, the coalgebra $\gr A$ becomes a graded Hopf algebra.
The projection $\pi:\gr A\rightarrow H$ with kernel $\oplus_{n>0}\gr^n A$ splits the inclusion
$\iota: H\rightarrow\gr A$, i. e. $\pi\circ\iota=\id_H$. Then $R=\gr A^{\co\pi}= \{a\in\gr A|(\id\ot\pi)\D(a)=a\ot1\}$
is a braided Hopf algebra in $\ydh$ and $\gr A\simeq R\#H$ as Hopf algebras \cite{radford, Mj}.
Here $\#$ stands for the smash product and the smash coproduct structures.
Moreover, $\gr A$ and $R$ have the following properties, see \cite{andrussch2}:
\begin{enumerate}\label{eq: properties of R and gr A} \renewcommand{\theenumi}{\alph{enumi}}\renewcommand{\labelenumi}{(\theenumi)}
\item $R = \oplus_{n\geq 0} R^n$ is a braided graded Hopf algebra with $R^n =R\cap\gr^n A$.
\item $\gr^n A=R^n\#H$.
\item $R_0=R^0 =\ku1$ and $R^1=P(R)$.
\end{enumerate}

\bigbreak
Clearly, $A \in \hmh$ by left and right multiplication. By \cite[Thm. 5.9.c)]{ardizzonimstefan},
there exists a projection $\Pi:A\rightarrow H$ of $H$-bimodule coalgebras such that $\Pi_{|H}=\id_{H}$.
Hence $A$ is a Hopf bimodule over $H$ with coactions $\rho_L:=(\Pi\circ\id)\D$ and $\rho_R:=(\id\circ\Pi)\D$.
Thus by \cite[Lemma 1.1]{andrunatale}, $A_1=H\oplus P_1$ as Hopf bimodules over $H$ where
\begin{align}\label{eq:p1}
P_1:=A_1\cap\ker\Pi= \{a\in A|\D(a)=\rho_L(a)+\rho_R(a)\}.
\end{align}
Let $V= R^1\in \ydh$; $V$ is called the \emph{infinitesimal braiding} of $A$. There exists an isomorphism $\gamma: V\#H\rightarrow P_1$ of Hopf bimodules over $H$, by (b) above.

\begin{prop}\label{prop: A with chevalley property is a quotient}
Let $A$ be a Hopf algebra whose coradical is a Hopf subalgebra isomorphic to $H$. Then
$\gr A\simeq\BV(V)\#H$ if and only if there exists an epimorphism of Hopf algebras $\phi:T(V)\#H\rightarrow A$ such that $\phi_{|H}=\id_H$ and $\phi_{|V\#H}:V\#H\rightarrow P_1$ is an isomorphism of Hopf bimodules over $H$.
\end{prop}

\begin{proof}
Suppose first that $\gr A\simeq\BV(V)\#H$. Let $\gamma:V\#H\rightarrow P_1$ be an isomorphism of Hopf bimodules over $H$ as above. Then $\phi:T(V)\#H\rightarrow A$ defined by $\phi(h)=h$ and $\phi(v)=\gamma(v\#1)$ is a morphism of Hopf algebras. Since $A$ is generated by $A_1$ by \cite[Lemma 2.2]{andrussch2}, $\phi$ is an epimorphism.

Now, suppose that such a $\phi$ is given (for an arbitrary $V$). By \cite[Cor. 5.3.5]{mongomeri} $H=A_0$; hence $\gr A\simeq R\#H$. Since $\phi_{|V\#H}:V\#H\rightarrow P_1$ and $\gamma:R^1\#H\rightarrow P_1$ are isomorphisms of Hopf bimodules over $H$, $\cP(R) = R^1\simeq V$ in $\ydh$.  By the definition of Nichols algebra of $V$, we are left to show that $V$ generates $R$ but this follows from the fact that $\phi$ is surjective.
\end{proof}

\begin{obs}\label{obs:J2 go to A1} Let $\ad: A \to \End A$ be the adjoint representation, $\ad x(y) = x\_1 y \cS(x\_2)$, $x,y\in A$.
In the notation of Proposition \ref{prop: A with chevalley property is a quotient}, $\phi(h\cdot x)=\ad h(x)$ for all $h\in H$ and $x\in T(V)$. If $\cJ$ is the  defining ideal of $\BV(V)$ as above, then  $\phi(\cJ^2\#1)\subset A_1$ by Lemma \ref{le:component of degree 2 of ideal which defines the Nichols}.
\end{obs}

\section{Hopf algebras with coradical $\ku^{\Sn_3}$}

\subsection{Hopf algebras with coradical a dual group algebra} We start discussing Hopf algebras $A$ with coradical $\ku^G$, $G$ a finite group.
Let $(\delta_h)_{h\in G}$ be the basis of $\ku^G$ dual to the basis consisting of group-like elements of $\ku G$; $\delta_h(g) = \delta_{g,h}$,
where the last is the Dirac delta.  Then for all $h, t\in G$,
$$\D(\delta_h) = \sum_{t\in G} \delta_t \otimes \delta_{t^{-1}h}, \qquad \ad \delta_h(\delta_t) = \delta_{h,e}\delta_t.$$
If $M\in \hm$, then we set $M^g := \{x\in M|\delta_h\cdot x=\delta_{h,g}x\,\forall h\in G\}$, $g\in G$, and $\supp M := \{g\in G: M^g \neq 0\}$.
This applies to the $n$-th  term $A_n$ of the coradical filtration, a $\ku^G$-module via the adjoint action.

\begin{lema}\label{le:filtration of a Hopf algebra with coradical the dual of kG}
Let $A$ be a Hopf algebra with coradical $\ku^G$ and infinitesimal braiding $V = \oplus_{i\in I}M(g_i,\rho_i)$. Then
\begin{enumerate}\renewcommand{\theenumi}{\alph{enumi}}\renewcommand{\labelenumi}{(\theenumi)}
\item\label{item:a} $A_n^g\cdot A_m^h\subseteq A_{n+m}^{gh}$ for all $n,m\geq0$ and $g,h\in G$. Hence $A_n^g\in \kgmkg$.
\item\label{item:b} If $x_g\in A_n^g$ then $\delta_h x_g=x_g\delta_{g^{-1}h}$ for all $h\in G$.
\item\label{item:c} If  $x_g\in A_n^g$, $g\in G$, then
$\D(x_g)=\sum_{t\in G} (y_g^t\ot\delta_t + \delta_t\ot z_{t^{-1}gt}^t)+ w$ with $w\in\bigoplus_{s, t\in G}\bigoplus_{i=1}^{n-1}(A^{s}_{i}\ot A^{t}_{n-i})$ and $y_g^t, z_g^t\in A_{n}^g$.
\item\label{item:antipoda} If $g\in G$, then  $\cS (A_n^g) = A_n^{g^{-1}}$.
\item\label{item:d} $(\supp A_1)^{-1} =\bigcup_{i\in I} \cO_{g_i}\cup\{e\}$.
\item\label{item:e} If $A$ is finite-dimensional then $A_1^e=\ku^G$.
\end{enumerate}
\end{lema}

\begin{proof}
Let $x_g\in A_n^g$ and $x_h\in A_m^h$ then
$$\ad\delta_s(x_g x_h)=\sum_{t\in G}\ad\delta_t(x_g)\ad\delta_{t^{-1}s}(x_h)=\delta_{s,gh} x_g x_h,$$
since the only non-zero term occurs when $t=g$ and $t^{-1}s=h$. This implies \eqref{item:a}; note that $A_n^g$ is a $\ku^G$-bimodule
because $\ku^G= A_0^e$. Now \eqref{item:b} follows from
\begin{align*}
\delta_hx_g&=\sum_{s\in G}\delta_sx_g\e(\delta_{s^{-1}h})=\sum_{s\in G}\ad\delta_s(x_g)\delta_{s^{-1}h}=x_g\delta_{g^{-1}h}.
\end{align*}

\noindent By \eqref{eq: properties of coradical filtration}, we can write $\displaystyle \D(x_g)=\sum_{s,t\in G}(y_s^t\ot\delta_t + \delta_t\ot z_s^t) + w$
with $y_s^t,z_s^t\in A_n^{s}$ and $\displaystyle w\in\bigoplus_{\substack{s, t\in G\\ 1\le i \le n - 1}} (A^{s}_{i}\ot A^{t}_{n-i})$. If $w=w_1\ot w_2$, then $\displaystyle\tilde{w}=\sum_{f,h,s,t\in G}\delta_fw_1\cS(\delta_{h^{-1}g})\ot\ad\delta_{f^{-1}h}(w_2)$ also belongs to
$\bigoplus_{\substack{s, t\in G\\ 1\le i \le n - 1}} A^{s}_{i}\ot A^{t}_{n-i}$. Then

\begin{align*}
\D(x_g) &= \D(\ad\delta_g(x_g))=\sum_{f,h,s,t\in G}\delta_fy_s^t\cS(\delta_{h^{-1}g})\ot\ad\delta_{f^{-1}h}(\delta_t)\\
&+\sum_{f,h,s,t\in G}\delta_f\delta_t\cS(\delta_{h^{-1}g})\ot\ad\delta_{f^{-1}h}(z_s^t)+\tilde{w}\\
&=\sum_{h,s,t\in G}\delta_hy_s^t\cS(\delta_{h^{-1}g})\ot\delta_t
+\sum_{f,s,t\in G}\delta_{f}\delta_t\cS(\delta_{(fs)^{-1}g})\ot z_s^t+\tilde{w}\\
&= \sum_{s,t\in G}\ad\delta_g(y_s^t)\ot\delta_t+\sum_{s,t\in G}\delta_t\delta_{g^{-1}ts}\ot z_s^t +\tilde{w}\\
&=\sum_{t\in G}y_g^t\ot\delta_t+\sum_{t\in G}\delta_t\ot z_{t^{-1}gt}^t+\tilde{w}.
\end{align*}

The proof of \eqref{item:antipoda} is straightforward. We observed in Subsection \ref{subsection:coradical=Hopf} that
$A_1=\ku^G \oplus V\# \ku^G$ as Hopf bimodules. Thus $A_1^g = V^g\# \ku^G$, $g\neq e$, and $A_1^e = \ku^G \oplus V^e\# \ku^G$.
By \eqref{eq: yetter-drinfield over de dual of G}, \eqref{item:d} follows.

Let $K$ be the subalgebra generated by $A_1^e$; by \eqref{item:c} $A_1^e$ is a coalgebra and by \eqref{item:antipoda} it is stable under the antipode, hence $K$ is a Hopf subalgebra.
By \eqref{item:b} $\ku^G$ is a normal Hopf subalgebra of $K$ and if $\dim A<\infty$,
then we have the exact sequence of Hopf algebras $\ku^G\hookrightarrow K\twoheadrightarrow K/(\ku^G)^+K$,
see \cite{andrudevoto}. We claim that $K/(\ku^G)^+K=\ku$ and therefore $\ku^G=K$ by \cite{andrudevoto}.

Let $x_e \in B := V^e\#\ku^G$, identified with a subspace of $A_1^e$ as above.
Since $V\#\ku^G \simeq P_1 \subset \ker \varepsilon$, see \eqref{eq:p1}, we have $$A_1^e = \ku \delta_e \oplus (\sum_{t\neq e}\ku \delta_t) \oplus
B\text{ and } A_1^e \cap \ker \eps =  \sum_{t\neq e}\ku \delta_t \oplus B.$$

By \eqref{item:c}, $\D(x_e)\in A_1^e \ot\ku^G + \ku^G\ot A_1^e = (B \ot \ku^G) \oplus
(\ku^G \ot B) \oplus (\ku^G\ot \ku^G)$. Write correspondingly
$$\D(x_e)=\sum_{t\in G}(u_t\ot\delta_t + \delta_t\ot v_t) + \sum_{s,t\in G}b_{s,t}\delta_s  \ot \delta_t,$$
with $u_t, v_t \in B$ and $b_{s,t}\in \ku$.
Computing $(\e\ot\id)\D(x_e)$ and $(\id\ot\e)\D(x_e)$,
we get $x_e= v_{e} + \sum_{s\in G}b_{e,s}\delta_s = u_{e} + \sum_{s\in G}b_{s,e}$
and then $v_{e}=x_e=u_{e}$ and  $b_{e,s}=0=b_{s,e}$ for all $s\in G$.
Hence $\overline{x_e}$ is a primitive element in $K/(\ku^G)^+K$, therefore $\overline{x_e}=0$ since $\dim A < \infty$ and the claim is true.
\end{proof}

\subsection{Hopf algebras with coradical $\ku^{\Sn_n}$}
Let $\cO_2^n$ be the conjugacy class of $(12)$ in $\Sn_n$ and let $\sgn:C_{\Sn_n}(12) \rightarrow\ku$ be the restriction of the sign representation of $\Sn_n$.
Let $V_n = M((12),\sgn) \in \ydsn$, cf. Definition \ref{eq:definition of simple in YD group};
$V_n$ has a basis $(\xij{ij})_{(ij)\in\cO_{2}^n}$ such that
$\delta(\xij{ij})=(ij)\ot\xij{ij}$ and $g\cdot\xij{ij}=\sgn(g)x_{g(ij)g^{-1}}$, $g\in \Sn_n$.
By \eqref{eq: yetter-drinfield over de dual of G}, $V_n$ turns into an object in $\ydsnd$ with action and coaction given by
\begin{align*}
&\delta_h\cdot\xij{ij} =\delta_{h,(ij)}\,\xij{ij}&\mbox{and}&&
\lambda(\xij{ij})=\sum_{h\in G}\sgn(h)\delta_{h}\ot x_{h^{-1}(ij)h}.
\end{align*}

Let $\cJ_n$ be the ideal of relations of $\BV(V_n)$. The elements
\begin{align}
\label{eq:rels-powers}\xij{ij}^2&,\\
\label{eq:rels-ijkl}R_{(ij)(kl)}&:=\xij{ij}\xij{kl}+\xij{kl}\xij{ij},
\\
\label{eq:rels-ijik}R_{(ij)(ik)}&:=\xij{ij}\xij{ik}+\xij{ik}\xij{jk}+\xij{jk}\xij{ij}
\end{align}
for all $(ij),(lk),(ik)\in\cO_2^n$ such that $(lk)\neq(ij)\neq(ik)$ form a basis of $\cJ_n^2$.

\begin{obs} It was shown
by  \cite{milinskisch} for $n=3, 4$ and by \cite{grania} for $n=5$ that $\cJ_n$ is generated by these elements and $\BV(V_n)$ is
finite-dimensional.
\end{obs}

\bigbreak We first show that relations \eqref{eq:rels-ijkl} and \eqref{eq:rels-ijik} hold in any lifting of $\BV(V_n)$.
In what follows, $A$ is a Hopf algebra such that $\gr A\simeq\BV(V_n)\#\ku^{\Sn_n}$ and $\phi:T(V_n)\#\ku^{\Sn_n}\to A$
is as in Proposition \ref{prop: A with chevalley property is a quotient}.

\begin{lema}\label{lem: relations in some liftings of BVnxSn}
We have $\phi(R_{(ij)(lk)})=0$ for all $(ij)\neq(kl)\in\cO_2^n$.
If $A$ has finite dimension, then $\phi(\sum_{(ij)\in\cO_2^n}\xij{ij}^2)=0$.
\end{lema}

\begin{proof} Since $\phi(R_{(ij)(kl)})\in A_1^{(ij)(kl)}$ by Lemma \ref{le:component of degree 2 of ideal which defines the Nichols},
Remark \ref{obs:J2 go to A1} and Lemma \ref{le:filtration of a Hopf algebra with coradical the dual of kG}
\eqref{item:a}, we see that
$\phi(R_{(ij)(lk)}) = 0$ from Lemma \ref{le:filtration of a Hopf algebra with coradical the dual of kG}
\eqref{item:d}. Since $\sum_{(ij)\in\cO_2^n}\xij{ij}^2$ is primitive in $T(V)$ and spans the trivial Yetter-Drinfeld module,
we conclude that $\phi(\sum_{(ij)\in\cO_2^n}\xij{ij}^2)$ is primitive in $A$, hence it is 0.
\end{proof}

\bigbreak
\begin{definition}\label{def: cAa1a2}
Given $(a_1,a_2)\in\ku^2$, we denote by $\cA_{[a_1,a_2]}$ the algebra $(T(V_3)\#\ku^{\Sn_3})/\cI_{(a_1,a_2)}$ where $\cI_{(a_1,a_2)}$
is the ideal generated by
\begin{align}\label{eq:rels-ideal}
\begin{aligned}
&R_{(13)(23)},\\
&R_{(23)(13)}, \\
&\xij{13}^2-(a_1-a_2)(\delta_{(12)}+\delta_{(123)})-a_1(\delta_{(23)}+\delta_{(132)}),\\
&\xij{23}^2-a_2(\delta_{(13)}+\delta_{(123)})-(a_2-a_1)(\delta_{(12)}+\delta_{(132)}),\\
&\xij{12}^2+a_1(\delta_{(23)}+\delta_{(123)})+a_2(\delta_{(13)}+\delta_{(132)}).
\end{aligned}
\end{align}
\end{definition}

It is easy to see that  $\cI_{(a_1,a_2)}$ is a Hopf ideal, hence $\cA_{[a_1,a_2]}$ is a Hopf algebra. Clearly, $\cA_{[0,0]}\simeq \BV(V_3)\#\ku^{\Sn_3}$.

\medbreak
Recall that $\Gamma = \ku^{\times}\times\Sn_3$ acts on $\ku^2$ by \eqref{equ:action}. Here is the main result of this paper.
Notice that Lemma \ref{le:similar to trivial casiprimitive} is used in the proof to find out the deformed relations.

\medbreak
\begin{theorem}\label{main result}

\begin{enumerate}\renewcommand{\theenumi}{\alph{enumi}}\renewcommand{\labelenumi}{(\theenumi)}
  \item Let $A$ be a finite-dimensional non-semisimple Hopf algebra with coradical $\ku^{\Sn_3}$. Then
$A\simeq\cA_{[a_1,a_2]}$ for some $(a_1,a_2)\in\ku^2$.

\smallbreak  \item The coradical of $\cA_{[a_1,a_2]}$ is isomorphic to $\ku^{\Sn_3}$ and $\dim \cA_{[a_1,a_2]}= 72$ for all $(a_1,a_2)\in\ku^2$.

\smallbreak  \item $\cA_{[b_1,b_2]}\simeq\cA_{[a_1,a_2]}$ if and only if $[b_1,b_2]=[a_1,a_2]$.
\end{enumerate}
\end{theorem}

\begin{proof}
(a) If $\gr A\simeq R\#\ku^{\Sn_3}$, then $V=R^1\in\ydSndual{3}$ and $\BV(V)$ is a braided Hopf subalgebra of $R$ by \cite[Prop. 2.2]{andrussch3}.
Thus $V\simeq V_3$ because $\BV(V_3)$ is the only Nichols algebra of finite dimension in ${}_{\ku\Sn_3}^{\ku\Sn_3}\yd$ by \cite[Thm. 4.5]{AHS}.
Since $\BV(V_3)$ only depends on the braiding of $V$, we can deduce from \cite[Thm. 2.1]{andrusgrania3} that $R=\BV(V_3)$.

Let $\rho:\Sn_3 \to GL(2, \ku)$ be the irreducible representation defined in the canonical basis $(e_i)$ by
$$\rho(12)\cdot e_1=e_2, \qquad \rho(123)\cdot e_1=- e_2.$$
Let $f_{ji}$ be the matrix coefficients of $\rho$; thus
\begin{align*}
\left(
\begin{matrix}
f_{11} & f_{12}\\
f_{21} & f_{22}
\end{matrix}
\right)=\left(
\begin{matrix}
\delta_{e}+\delta_{(13)}-\delta_{(23)}-\delta_{(132)} & \delta_{(12)}-\delta_{(23)}+\delta_{(123)}-\delta_{(132)}\\
\delta_{(12)}-\delta_{(13)}-\delta_{(123)}+\delta_{(132)} & \delta_{e}-\delta_{(13)}+\delta_{(23)}-\delta_{(123)}\\
\end{matrix}
\right)
\end{align*}

Considering $T(V_3) \in {}_{\ku\Sn_3}^{\ku\Sn_3}\yd$, we see that the elements
$c_1 :=\xij{13}^2-\xij{12}^2$ and $c_2 :=\xij{23}^2-\xij{12}^2$ of $T(V)$ span a Yetter-Drinfeld submodule
isomorphic to $M(e,\rho)$, via $e_i\mapsto c_i$, $1\le i \le 2$.
Hence this subspace becomes a submodule in $\ydSndual{3}$, with
coaction $\lambda$ given by $\lambda(c_i) = \sum_{j=1}^2 e_{ij} \ot c_j$, where $e_{ij}:=\cS^{-1}(f_{ji})$ for $i,j=1,2$.
Then $$\D\phi(c_i)=\phi(c_i)\ot1+\sum_{j}e_{ij}\ot\phi(c_j).$$
Since $\phi(c_1),\,\phi(c_2)\in A_1^e=\ku^{\Sn_3}$, we conclude that they belong to $\ku 1 \oplus C$, where $C$ is the simple subcoalgebra of rank $2$
spanned the $e_{ij}$'s. By Lemma \ref{le:similar to trivial casiprimitive},
there exist $a_1, a_2\in\ku$ such that $\phi(c_i)=a_i-\sum_{j=1}^2a_je_{ij}$ for $i=1,2$ . Note that the elements \eqref{eq:rels-ideal} form a basis of
$$M_{(a_1,a_2)}=\langle c_i-a_i+\sum_{j}a_je_{ij},\, R_{(13)(23)},\, R_{(23)(13)}, \sum_{(ij)\in\cO_2^3}\xij{ij}^2\rangle,$$
which is a coideal contained in $\ker\e$. Then the ideal $\cI_{(a_1,a_2)}$ generated by $M_{(a_1,a_2)}$ is a Hopf ideal
and $\cI_{(a_1,a_2)}\cap(\ku\oplus V_3)\#\ku^{\Sn_3}=0$ because $\cI_{(a_1,a_2)}\subset\ker\phi$.
Thus $\gr(T(V_3)\#\ku^{\Sn_3}/\cI_{(a_1,a_2)})\simeq R\#\ku^{\Sn_3}$ and $R$ is generated by $\phi(V_3)$.
Hence $R\simeq T(V_3)/I$ with $I\subseteq\cJ_3$. Moreover, the generators of $\cJ_3$ are contained in $I$
by the definition of $\cI_{(a_1,a_2)}$ and therefore $I=\cJ_3$. Then  $\dim A=\dim (T(V_3)\#\ku^{\Sn_3}/\cI_{(a_1,a_2)})$,
that is, $\ker\phi=\cI_{(a_1,a_2)}$.

\bigbreak We claim that $\cB=\{x\delta_g|x\in B,\,g\in\Sn_3\}$ is a basis of $\cA_{[a_1,a_2]}$ where
$$
B=\left\{
\begin{matrix}
1, &\xij{13}, &\xij{13}\xij{12}, &\xij{13}\xij{12}\xij{13}, &\xij{13}\xij{12}\xij{23}\xij{12},\\
   &\xij{23}, &\xij{12}\xij{13}, &\xij{12}\xij{23}\xij{12},\\
   &\xij{12}, &\xij{23}\xij{12}, &\xij{13}\xij{12}\xij{23},\\
   &          &\xij{12}\xij{23}
\end{matrix}
\right\}
$$
and therefore (b) follows. Next, we sketch a proof of the claim using the Diamond Lemma \cite{diamondlemma}.

We need to introduce more relations which are deduced from
\eqref{eq:rels-ideal}. We write the relations of the form $R=f$ with $R$ a monomial of $\cA_{[a_1,a_2]}$ and $f\in\ku\cB$. The new list of relations is
\begin{align*}
1=&\sum_{g\in\Sn_3}\delta_g, \quad \delta_g\delta_h=\delta_{g,h}\delta_g, \quad \delta_g\xij{ij}=\xij{ij}\delta_{(ij)g},\\
\xij{13}^2=&(a_1-a_2)(\delta_{(12)}+\delta_{(123)})+a_1(\delta_{(23)}+\delta_{(132)}),\\
\xij{23}^2=&a_2(\delta_{(13)}+\delta_{(123)})+(a_2-a_1)(\delta_{(12)}+\delta_{(132)}),\\
\xij{12}^2=&-a_1(\delta_{(23)}+\delta_{(123)})-a_2(\delta_{(13)}+\delta_{(132)}),\\
\xij{13}\xij{23}=&-\xij{23}\xij{12}-\xij{12}\xij{13},\\
\xij{23}\xij{13}=&-\xij{12}\xij{23}-\xij{13}\xij{12},\\
\xij{12}\xij{13}\xij{12}=&\xij{13}\xij{12}\xij{13}+\xij{23}a_1,\\
\xij{23}\xij{12}\xij{23}=&\xij{12}\xij{23}\xij{12}-\xij{13}a_2\,\mbox{ and}\\
\xij{23}\xij{12}\xij{13}=&\xij{13}\xij{12}\xij{23}+\xij{12}\Omega
\end{align*}
where $\Omega=(a_2-a_1)(\dij{(12)}-\dij{e})+a_1(\dij{(13)}-\dij{(132)})-a_2(\dij{(23)}-\dij{(123)})$. Following the Diamond Lemma, we have to show that if $X,Y,Z$ are monomial
in $A\setminus\{1\}$ such that $R_1=XY$ and $R_2=YZ$ then $f_1Z$ and $Xf_2$ can be reduced to a same element in $\ku\cB$ using the before list of relations --
$XYZ$ is called ``overlap ambiguity`` and if the above is true it said that the ambiguity is ''resoluble``; it is defined in \cite{diamondlemma} other type
of ambiguity but this does not happen in our case. Calculate the ambiguities and show that these are resoluble is a tedious but straightforward computation.

\bigbreak (c) We denote the elements in $\cA_{[a_1,a_2]}$ and $\cA_{[b_1,b_2]}$ by $\ao{x}\delta_g$ and $\bo{x}\delta_g$ for $x\in T(V_3)$, $g\in\Sn_3$.
Let $\Theta:\cA_{[b_1,b_2]}\rightarrow\cA_{[a_1,a_2]}$ be an isomorphism of Hopf algebras. Since $(\Theta_{|\ku^{\Sn_3}})^*$
induces a group automorphism of $\Sn_3$, $\Theta(\delta_g)=\delta_{\theta g\theta^{-1}}$ for some $\theta\in\Sn_3$.
By the adjoint action of $\ku^{\Sn_3}$,
$\Theta({}^b\overline{\xij{ij}})\in\sum_{g\in\Sn_3}\lambda_g\ao{x_{\theta(ij)\theta^{-1}}}\delta_g$
with $\lambda_g\in\ku$ but since $\Theta$ is a coalgebra morphism, all these $\lambda_g$'s  are equal.
Then for each $(ij)\in\cO_2^3$ there exists $\mu_{(ij)}\in\ku^{\times}$ such that $\Theta(\bo{\xij{ij}})=\mu_{(ij)}\ao{x_{\theta(ij)\theta^{-1}}}$. Also,
$$0=\ao{\mu_{(23)}\mu_{(13)}\xij{23}\xij{13}+\mu_{(13)}\mu_{(12)}\xij{13}\xij{12}+\mu_{(12)}\mu_{(23)}\xij{12}\xij{23}}$$ because it is equal to
$\Theta(\bo{R_{(13)(23)}})$ if $\theta=(12)$ or  to $\Theta(\bo{R_{(23)(13)}})$ if $\theta=(123)$.
Then $\mu_{(23)}\mu_{(13)}=\mu_{(13)}\mu_{(12)}=\mu_{(12)}\mu_{(23)}$ because otherwise $\dim\cA(a_1,a_2)$
would be less than $72$. Since $\mu_{(ij)}\neq0$ it results that
$\Theta(\bo{\xij{ij}})= \mu\,\ao{x_{\theta(ij)\theta^{-1}}}$ with $\mu\in\ku^{\times}$
for all $(ij)\in\cO_2^3$. Since $(\delta_g)_{g\in\Sn_3}$ is  linearly independent, we obtain
that $(b_1,b_2)=(a_1,a_2)\triangleleft(\mu^2,\theta)$ by the equality $\Theta(\bo{\xij{12}^2})=\ao{x_{\theta(12)\theta^{-1}}^2}$.
Conversely, given $(\mu, \theta)\in\Gamma$, the map $\Theta_{\mu, \theta}:T(V_3)\#\ku^{\Sn_3}\rightarrow T(V_3)\#\ku^{\Sn_3}$
defined by $\Theta_{\mu, \theta}(\delta_g)=\delta_{\theta g\theta^{-1}}$, $\Theta_{\mu, \theta}(\xij{ij})
=\mu x_{\theta (ij)\theta^{-1}}$, for $g\in\Sn_3$. $(ij)\in\cO_2^3$, is a Hopf algebra isomorphism such
that $\Theta_{\mu,\theta}(\cI_{(a_1,a_2)})=\cI_{(a_1,a_2)\triangleleft(\mu^2,\theta)^{-1}}$.
\end{proof}

\end{document}